\documentclass[letterpaper, 12pt]{amsart}

\usepackage{latexsym,mathrsfs,amscd,amsmath,verbatim,amssymb,calc,enumerate,amsxtra,
ifpdf}
\usepackage[all]{xy}
\ifpdf%
  \usepackage{hyperref}
\else%
  \usepackage[hypertex]{hyperref}
\fi%

\allowdisplaybreaks[1]

\theoremstyle{plain}
\newtheorem{theorem}{Theorem}
\newtheorem{corollary}[theorem]{Corollary}

\newtheorem*{theorem*}{Theorem}

\theoremstyle{remark}

\theoremstyle{definition}

\newtheorem{question}[theorem]{Question}

\numberwithin{equation}{theorem}

{\noindent\ignorespaces}%
{\par\noindent%
\ignorespacesafterend}

\newcommand{\ZZ}{\mathbb{Z}}
\newcommand{\m}{\mathfrak{m}}

\newcommand{\e}{\operatorname{e}}

\begin{document}

\title[Periodic terms of Hilbert--Kunz functions]{A Hilbert--Kunz function with a periodic term that has a given period}
\author{Robin Baidya}
\email{rbaidya@utk.edu}
\address{Department of Mathematics, The University of Tennessee, Knoxville, Tennessee 37996}
\curraddr{}
\thanks{The author would like to thank Florian Enescu for sharing the interesting example of Monsky that inspired this paper.  The author would also like to thank Yongwei Yao, Felipe P\'erez, Siang Ng, and Irina Ilioaea for helpful comments during a presentation of these results at Georgia State University.}
\date{\today}
\begin{abstract}
A result of Monsky states that the Hilbert--Kunz function of a one-dimensional local ring of prime characteristic has a term $\phi$ that is eventually periodic.  For example, in the case of a power series ring in one variable over a prime-characteristic field, $\phi$ is the zero function and is therefore immediately periodic with period 1.  In additional examples produced by Kunz and Monsky, $\phi$ is immediately periodic with period~2.  We show that, for every positive integer~$\pi$, there exists a ring for which $\phi$ is immediately periodic with period~$\pi$.
\end{abstract}
\commby{}
\maketitle

\setcounter{section}{+0} 

Let $(R,\m)$ be a one-dimensional commutative Noetherian local ring of prime characteristic~$p$.  In this note, we study the \textit{Hilbert--Kunz function $HK_{\m,R}$}, which sends every nonnegative integer $e$ to the length of $R/\m^{[p^e]}$ over~$R$, where $\m^{[p^e]}:=(r^{p^e}:r\in\m)$.

The general form of $HK_{\m,R}$ is known:  Monsky has proved that there exist a positive integer $\e_{HK}(\m,R)$, called the \textit{Hilbert--Kunz multiplicity of~$R$}, and an eventually periodic function $\phi_{\m,R}$ such that 
\[
HK_{\m,R}(e)=\e_{HK}(\m,R)\cdot p^e-\phi_{\m,R}(e)
\]
for every nonnegative integer~$e$~\cite[Theorem~3.10]{Mon}.

Also, there are algorithms for specialized settings:  Kreuzer has covered the case in which $R=\hat{S}^M$, where $S$ is a standard graded algebra over a field and $M$ is the maximal homogeneous ideal of~$S$~\cite[Algorithm~5.4]{Kre}.  Upadhyay has obtained a formula valid in the event that $S$ is a polynomial ring over a field modulo a binomial with zero constant term (and $R=\hat{S}^M$ with $M$ the maximal homogeneous ideal of~$S$ once again)~\cite[Theorem~26]{Upa}.

Despite these advances, explicit values of $HK_{\m,R}$ are still scarce.  Moreover, when values are available, $\phi_{\m,R}$ does not differ much from one example to the next:  For instance, in the six computations achieved by Kreuzer~\cite[Examples~6.1--6.6]{Kre}, the function $\phi_{\m,R}$ is eventually constant in four cases, and in the remaining two examples, which are attributable to Kunz~\cite[Example~4.6b]{Kun1} and Monsky~\cite[page~46]{Mon}, the function $\phi_{\m,R}$ is immediately periodic with period~2.

In response to the lack of diversity in concrete examples of~$\phi_{\m,R}$, we show here that, for every positive integer~$\pi$, there exists a ring $(R,\m)$ for which $\phi_{\m,R}$ is immediately periodic with period~$\pi$.  To accomplish our goal, we first provide an infinite family of rings whose Hilbert--Kunz functions we can compute explicitly.  To illustrate the content of our result, we recover Monsky's example with period 2 and give new examples with periods 3 and~4.  We then appeal to a classical observation of Dirichlet to prove our main theorem.

\begin{theorem}\label{prelim}
  Let $(R,\m):=k[[x,y]]/(x^n-y^n)$, where $x$ and $y$ are indeterminates, $k$ is a field of prime characteristic~$p$, and $n$ is an integer that is greater than $1$ and not divisible by~$p$.  Let $\pi$ denote the period eventually exhibited by~$\phi_{\m,R}$, and let $\omega$ stand for the order of $p+n\ZZ$ in the multiplicative group of units of $\ZZ/n\ZZ$.  Then the following statements hold:
\begin{enumerate}
\item  Let $e$ be a nonnegative integer, and let $b$ be an integer satisfying $p^e\equiv b$ \textnormal{(mod $n$)} and $1\leqslant b\leqslant n-1$.  Then $HK_{\m,R}(e)=np^e-b(n-b)$. 
\item  $\phi_{\m,R}$ is immediately periodic with $\pi\in\{\omega/2,\omega\}$.
\item  $\pi=\omega/2$ if and only if $\omega$ is even and $p^{\omega/2}\equiv -1$\textnormal{ (mod }$n$\textnormal{)}.
\end{enumerate}
\end{theorem}

\begin{proof} (1)  Let $I$ denote the ideal $(x^q,y^q,x^n-y^n)$ of~$k[x,y]$, where~$q:=p^e$, and let $T:=k[x,y]/I$.  Then $HK_{\m,R}(e)=\dim_k(T)$.  Hence, we must show that $\dim_k(T)=nq-b(n-b)$.

If $q<n$, then $I=(x^q,y^q)$, and so $\dim_k(T)=q^2=nq-b(n-b)$, completing the proof.  Of course, $q\neq n$ since, by hypothesis, $n$ is not divisible by~$p$.  Therefore, we may assume that $q>n$.

Next, we establish that $G:=\{x^by^{q-b},y^q,x^n-y^n\}$ is a Gr\"{o}bner basis for $I$ with respect to~lex($x>y$), the lexicographic monomial order induced by the relation $x>y$.  Certainly, $I$ is equal to the ideal of $k[x,y]$ generated by $G$ since
\[
x^q-x^by^{q-b}=(x^{q-n}y^0+x^{q-2n}y^n+\cdots+x^by^{q-b-n})(x^n-y^n)\in I\cap (G)
\]
and since $y^q,x^n-y^n\in I\cap(G)$ by definition.  To verify that $G$ satisfies Buchberger's Criterion, we note that
\[
S(x^by^{q-b},y^q)=y^b(x^by^{q-b})-x^b(y^q)=0
\]
and that
\[
S(x^by^{q-b},x^n-y^n)=x^{n-b}(x^by^{q-b})-y^{q-b}(x^n-y^n)=y^{q+n-b}
\]
and
\[
S(y^q,x^n-y^n)=x^n(y^q)-y^q(x^n-y^n)=y^{q+n}
\]
are divisible by~$y^q$ and no other member of~$G$.  Hence $G$ is a Gr\"{o}bner basis for $I$ with respect to lex($x>y$).

Now let $A$ be the set of monomials in $k[x,y]$ not divisible by any of the leading terms of the members of~$G$, and let $^-$ denote the natural projection of $k[x,y]$ onto~$T$.  Then $\{\overline{f}:f\in A\}$ is a basis for $T$ over~$k$.  Hence $\dim_k(T)=|A|=nq-b(n-b)$.

(2)  Write $\phi:=\phi_{\m,R}$.  Then Part~(1) implies that $\phi$ is immediately periodic with $\pi$ dividing~$\omega$.  Hence $\phi(\pi)=\phi(0)$.  Part (1) now gives us $p^{\pi}(n-p^{\pi})\equiv 1(n-1)$ (mod $n$), which implies that $p^{2\pi}\equiv 1$ (mod $n$), which in turn implies that $\omega$ divides~$2\pi$.  Hence $\pi\in\{\omega/2,\omega\}$.

(3)  Let $\phi:=\phi_{\m,R}$, and suppose that $\pi=\omega/2$.  Since $\pi$ is an integer, $\omega$ must be even.  Let $b$ be an integer such that $p^{\pi}\equiv b$ (mod $n$) and $1\leqslant b \leqslant n-1$.  Since Part~(2) yields $\phi(0)=\phi(\pi)$, we get $1(n-1)= b(n-b)$.  Solving for $b$, we get $b\in\{1,n-1\}$.  If $b=1$, then $p^{\pi}\equiv b\equiv 1$ (mod $n$), and so $\omega$ divides~$\pi=\omega/2$, a contradiction.  Hence $b=n-1$, and so $p^{\omega/2}=p^{\pi}\equiv b\equiv -1$ (mod $n$).

Conversely, suppose that $\omega$ is even and that $p^{\omega/2}\equiv -1$ (mod $n$).  Let $e$ be an integer with $0\leqslant e \leqslant (\omega/2)-1$, and let $c$ be an integer satisfying $p^e\equiv c$ (mod $n$) and $1\leqslant c \leqslant n-1$.  Since $p^{\omega/2}\equiv -1$ (mod $n$) by hypothesis, $p^{[e+(\omega/2)]}\equiv -c \equiv n-c$ (mod $n$).  So, by Part (1),
\[
\phi(e+(\omega/2))=(n-c)[n-(n-c)]=c(n-c)=\phi(e).
\]
We have thus shown that $\pi$ divides $\omega/2$.  Now, by Part~(2), we get $\pi=\omega/2$.
\end{proof}

Theorem~\ref{prelim} recovers the following example of Monsky in which $\pi=\omega/2$:

\begin{corollary}[{Monsky~\cite[page~46]{Mon}}]
	Let $(R,\m)=k[[x,y]]/(x^5-y^5)$, where $x$ and $y$ are indeterminates and $k$ is a field of characteristic~$p\equiv \pm 2$\textnormal{ (mod 5)}.  Then, for every nonnegative integer $e$, we have
\[
HK_{\m,R}(e)
=\begin{cases}
5p^e-4 & \textnormal{if $e$ is even;}\\
5p^e-6 & \textnormal{if $e$ is odd.}\\
\end{cases}
\]
Hence $\pi=2=\omega/2$.
\end{corollary}

The next example presents a case in which $\pi=3=\omega$.

\begin{corollary}
	Let $(R,\m)=k[[x,y]]/(x^7-y^7)$, where $x$ and $y$ are indeterminates and $k$ is a field of characteristic~$p\equiv 2$\textnormal{ (mod 7)}. Then, for every nonnegative integer $e$, we have
	\[
	HK_{\m,R}(e)
	=\begin{cases}
	7p^e-6 & \textnormal{if $e\equiv 0$ (mod 3);}\\
	7p^e-10 & \textnormal{if $e\equiv 1$ (mod 3);}\\
	7p^e-12 & \textnormal{if $e\equiv 2$ (mod 3).}\\ 
	\end{cases}
	\]
	Hence $\pi=3=\omega$.
\end{corollary}

Our final example illustrates that the congruence condition in Part (3) of Theorem~\ref{prelim} is not redundant:

\begin{corollary}
	Let $(R,\m)=k[[x,y]]/(x^{15}-y^{15})$, where $x$ and $y$ are indeterminates and $k$ is a field of characteristic~$p\equiv\pm 2$\textnormal{ (mod 15)}. Then, for every nonnegative integer $e$, we have
	\[
	HK_{\m,R}(e)
	=\begin{cases}
	15p^e-14 & \textnormal{if $e\equiv 0$ (mod 4);}\\
	15p^e-26 & \textnormal{if $e\equiv 1$ (mod 4);}\\
	15p^e-44 & \textnormal{if $e\equiv 2$ (mod 4);}\\	
	15p^e-56 & \textnormal{if $e\equiv 3$ (mod 4).}\\	
	\end{cases}
	\]
	Hence $\pi=4=\omega$.
\end{corollary}

We now recall a result of Dirichlet and prove our main theorem.

\begin{theorem}[{Dirichlet~\cite{Dir}}]\label{Dir}
  Let $r$ and $s$ be relatively prime positive integers.  Then there exist infinitely many primes congruent to $r$ modulo~$s$.
 \end{theorem}

\begin{theorem}[Main Theorem]
	Let $\pi$ be a positive integer.  Then there is a one-dimensional commutative Noetherian local ring $(R,\m)$ of prime characteristic such that $\phi_{\m,R}$ is immediately periodic with period~$\pi$.
\end{theorem}

\begin{proof}
	By Theorem~\ref{Dir}, there exists a prime $n\equiv 1$ (mod $2\pi$).  Since the multiplicative group of units $U$ of  $\ZZ/n\ZZ$ is cyclic of order~$n-1$ and since $n-1$ is divisible by~$2\pi$, there exists an integer $p$ such that $p+n\ZZ$ is an element of $U$ of order~$\omega:=2\pi$.  By Theorem~\ref{Dir}, we may take $p$ to be prime.  Since $U$ is cyclic of even order, $U$ contains a unique involution, namely, $-1+n\ZZ$.  Hence $p^{\omega/2}\equiv -1$ (mod $n$).  Now let $R:=k[[x,y]]/(x^n-y^n)$, where $x$ and $y$ are indeterminates and $k$ is a field of characteristic~$p$.  Then, by Theorem~\ref{prelim}, the function $\phi_{\m,R}$ is immediately periodic with period~$\omega/2=\pi$.
\end{proof}

In light of our observations, the following question naturally arises:

\begin{question}
	Let $\phi$ be a periodic function from the set of nonnegative integers to the set of all integers.  Does there exist a one-dimensional commutative Noetherian local ring $(R,\m)$ of prime characteristic such that $\phi(e)=\phi_{\m,R}(e)$ for all sufficiently large integers~$e$?
\end{question}

\bibliographystyle{amsplain}
\bibliography{HK}

\end{document}